\documentclass[11pt]{amsart}
\usepackage{latexsym}
\usepackage{color}
\usepackage[pdfstartview=FitH]{hyperref}
\usepackage{tikz-cd}
\usepackage{float}
\usepackage{soul}
\usepackage{subcaption}
\usepackage{adjustbox}
\usepackage{wasysym}

\newcommand{\da}[1]{{\color{blue}{ #1}}}

\newtheorem{theorem}{Theorem}[section]
\newtheorem*{theorem*}{Theorem}
\newtheorem{lemma}[theorem]{Lemma}
\newtheorem{definition}[theorem]{Def{}inition}

\newtheorem{proposition}[theorem]{Proposition}
\newtheorem*{proposition*}{Proposition}
\newtheorem{corollary}[theorem]{Corollary}
\newtheorem*{corollary*}{Corollary}

\newtheorem{remark}{Remark}
\newcommand{\iu}{\mathrm{i}\mkern1mu}
\def \<{\langle}
\def \>{\rangle}

\def \a{\alpha }

\def \l{\lambda }

\def \ps {\Psi}

\def \d {\delta}

\def \cla{c_{L,\a}}
\def \Cla{C_{L,\a}}
\DeclareMathOperator{\Ker}{Ker}

\def \ch{\mbox{char}_q}

\newcommand{\SH}{\mathcal{SH}}

\newcommand{\bea}{\begin{eqnarray}}
\newcommand{\eea}{\end{eqnarray}}
\newcommand{\be}{\begin {equation}}
\newcommand{\ee}{\end{equation}}
\newcommand{\g}{\frak g}

\newcommand{\vpr}{v_{p,r}}
\newcommand{\Z}{\mathbb Z}
\newcommand{\N}{{\mathbb Z}_{>0} }
\newcommand{\Zp}{{\mathbb Z}_{\ge 0} }
\newcommand{\C}{\mathbb C}
\newcommand{\Q}{\mathbb Q}

\usepackage{comment}
 
\newcommand{\ds}{\displaystyle}
\newcommand{\nn}{\nonumber \\}
\newcommand{\ec}{\Circle}
\newcommand{\bcr}{\CIRCLE}
 
\usepackage{amssymb}

\begin{document}

\title[$N=1$ Heisenberg-Virasoro algebra at level 0]{The $N=1$ super Heisenberg-Virasoro vertex algebra at level zero}
 
\author[]{Dra\v zen  Adamovi\' c}
\author[]{Berislav Jandri\' c}  
\author[]{Gordan Radobolja}
 
\date{\today}
\begin{abstract}
We study the representation theory of the $N=1$ super Heisenberg-Virasoro vertex algebra at level zero, which extends  the previous work on the Heisenberg-Virasoro vertex algebra \cite{B}, \cite{AR1} and \cite{AR2} to the super case. We calculated all characters of irreducible highest weight representations by investigating certain Fock space representations. Quite surprisingly, we found that the maximal submodules of certain Verma modules are generated by subsingular vectors. The formulas for singular and subsingular vectors are obtained using screening operators appearing in a study of certain logarithmic vertex algebras \cite{AdM-2010}.
\end{abstract}
\maketitle
\tableofcontents

\section{Introduction}

The Heisenberg-Virasoro vertex algebra is a generalisation of the Virasoro and the Heisenberg vertex algebra. It appears as a vertex subalgebra of af{}f{}ine vertex-algebras and $\mathcal W$--algebras. It is shown in  \cite{ACKP}, \cite{GW} that in the case of non-zero level, the Heisenberg-Virasoro vertex algebra is isomorphic to the tensor product of the Heisenberg and the Virasoro vertex algebra. But it is relatively less known that in the case of level zero, the Heisenberg-Virasoro vertex algebra has completely dif{}ferent structure. The study of the level-zero case was initiated by Y.\ Billig in \cite{B}. Among other results, Billig calculated characters of all irreducible, highest weight modules.
In our previous papers \cite{AR1}-\cite{AR2}, we extended Billig's work by applying recent constructions in vertex-algebra theory and conformal f{}ield theory such as fusion rules, logarithmic modules, Whittaker modules, screening operators. We have demonstrated that the Heisenberg-Virasoro vertex algebra is also an interesting example of a vertex algebra which allows us to understand more general non-rational vertex algebras. 

A natural problem is to try to f{}ind its super-analogue. This leads to the $N=1$ super Heisenberg-Virasoro vertex algebra. This vertex algebra was introduced in \cite{AJR}, using the notion of conformal Lie superalgebras.
The universal $N=1$ Heisenberg-Virasoro algebra is denoted by $V^\SH(c_L,c_\a,\cla)$, and is generated by two even f{}ields (the Virasoro f{}ield $L(z)$ and the Heisenberg f{}ield  $\a(z)$) and by two odd f{}ields $G(z)$  and $\Psi(z)$.
Similarly to the non-super case, in the case of a non-zero level (i.e., when the central element  $C_\a$ acts non-trivially), this vertex algebra is isomorphic to a tensor product of the Neveu-Schwarz vertex superalgebra and the Clif{}ford-Heisenberg algebra $SM(1)$ (cf.\ \cite{AJR}). In this paper, we start with a detailed study of the vertex algebra $V^\SH(c_L, \cla):=V^\SH(c_L,0, \cla)$, i.e.\ the level zero case.
  
We plan to extend results from \cite{B}, \cite{AR1} and \cite{AR2} to the super case. In the present paper, we calculated characters of all irreducible, highest weight modules and investigated the structure of the Verma modules and Fock spaces for $V^\SH(c_L, \cla)$. 
   
Let us point out one main dif{}ference in comparison with the non-super case. Certain Verma modules for $V^\SH(c_L, \cla)$ contain subsingular vectors which are not present in non-super case. Due to the existence of such vectors, description of the maximal submodule of Verma modules is more complicated in the super case, and Billig's method cannot be easily generalised.
Our method for calculating the character is based on free-f{}ield realisation, and formulas for singular/subsingular vectors. We show that in the "half cases", the irreducible highest weight modules are realised as submodules of certain Fock spaces. Combining this approach with explicit formulas for singular and subsingular vectors in Fock spaces, we are able to describe all characters.

\subsection*{Free-f{}ield realisation and screening operators}
In \cite{AJR} (see also \cite{J}), we obtained a free f{}ield realisation of $V^\SH(c_L,\cla)$ as a subalgebra of the tensor product of a rank two Heisenberg algebra $M(1)$ (generated by $c(z)$ and $d(z)$), and a fermionic algebra $F^{(2)}$ generated by $\Psi^+(z)$ and $\Psi^-(z)$. The following f{}ields
\bea
\alpha &=& - \cla c(-1) \nonumber \\
\tau &=&\sqrt{2}\left(\frac{1}{2} c(-1) \Psi^+(-\frac{1}{2}) + \frac{1}{2} d(-1) \Psi^- (-\frac{1}{2}) + \frac{c_L - 3}{12} \Psi^-(-\frac{3}{2}) - \Psi^+(-\frac{3}{2})\right) \nonumber \\
\omega &=& \frac{1}{2} c(-1) d(-1) + \frac{c_L -3}{24} c(-2) - \frac{1}{2} d(-2) + \omega_{fer} \nonumber \\
\Psi &=& -\sqrt{2} \cla\Psi^- (-\frac{1}{2})\nonumber 
\eea
generate $V^\SH(c_L,\cla)$.

In this paper we discuss a free-f{}ield realisation of highest weight modules for $V^\SH(c_L,\cla)$, which appear to be certain $V^\SH(c_L,\cla)$--submodules of Fock modules over the Clif{}ford-Heisenberg vertex algebra $M(1) \otimes F^{(2)}$. 

We introduce a suitable parametrisation (\ref{param}) of highest weights:
\[
\vpr:=e^{-\frac{p+1}{2}\overline d+r c},\qquad \text{where}\qquad\overline d=d-\tfrac{c_L-3}{12}c
\]
is highest weight vector of highest weight 
\[
(h_{p,r},h_{\a}) = \left((1-p^2)\frac{c_L-3}{24}-rp,(1+p)\cla\right).
\]

If $M(h,h_\a)$ is a highest weight module of highest weight $(h_{p,r},(1+p)h_\a)$ we denote it by $M[p,r]$.\\
We def{}ine the Fock space $\mathcal F_{p,r} := (M(1) \otimes F^{(2)} ). v_{p,r}$, and consider it as a $V^\SH(c_L,\cla)$-module.

We prove in Proposition \ref{realizacija1} that all irreducible highest weight modules such that $p\ne 0$, can be realised as subquotients of the Fock spaces.  More precisely:
\begin{itemize}
\item If $p \notin \Z_{<0}$, then $V[p,r] \cong  \mathcal F_{p,r} $ (cf.\ Proposition \ref{realizacija_Verma}).
\item If $p\in\Z_{<0}$, then $L[p,r] \cong \langle\vpr\rangle \subset \mathcal F_{p,r}$ (cf.\ Proposition \ref{ired-pmanji}).
\end{itemize}

Let $a=\Psi^-(-\tfrac{1}{2})e^{\tfrac{c}{2}}$. We show that $Q=a_0$ is a nilpotent screening operator which commutes with the action of $V^\SH(c_L,\cla)$.
\begin{itemize}
\item If $p\in\N$ is odd, $Q v_{p,r-\tfrac{1}{2}}$ is a singular vector in $V[p,r]$ (Theorem \ref{podd}).
\end{itemize}

In order to construct  other singular and subsingular vectors in Verma modules and Fock spaces, we need more complicated screening operators. For that purpose, we consider the operators (introduced in \cite{AdM-2010}):
\[
S=\sum_{i>0}\frac{1}{i}a_{-i}a_i,\qquad\text{and}\qquad S^{tw}=\sum_{i>0}\frac{1}{i+\tfrac{1}{2}}a_{-i-\tfrac{1}{2}}a_{i+\tfrac{1}{2}},
\]
which were used for studying logarithmic vertex operator algebras. In the present paper we show that
\begin{itemize}
\item $\mathcal G=e^c_0-S$ and $\mathcal G^{tw}=e^c_0-S^{tw}$ are screening operators (Theorem \ref{screeninzi}). 
\item $(\mathcal G^{tw})^n v_{p,r-n}$ are singular vectors in $V[p,r]$ when $p>0$ is even (Theorem \ref{sing-p-even-1}).
\item $\mathcal G^n Q v_{p,r-n-1/2}$ are singular vectors, and $\mathcal G^n v_{p,r-n}$ are subsingular vectors in $V[p,r]$ when $p>0$ is odd (Theorem \ref{sing-p-odd-1}).
\end{itemize}

\subsection*{Characters of $L[p,r]$ }

We use the free f{}ield realisation to f{}ind the $q$-characters of irreducible highest weight modules $L[p,r]$ (Theorem \ref{xx}). 

\begin{theorem} \label{xx-uvod} Assume that $p \in {\Z}\setminus \{ 0\} $ and $r \in {\C}$. Then we have:
\begin{align}
\ch L[p,r] &= q^{h_{p,r}} (1 -q ^{\tfrac{\vert p\vert}{2}})   \prod_{k=1} ^{\infty} \frac{ (1 + q^{k-1/2}) ^2}{(1-q^k) ^2},\quad\text{if }p\text{ is odd;}\nn
\ch L[p,r] &= q^{h_{p,r}} (1 -q ^{\vert p\vert   })   \prod_{k=1} ^{\infty} \frac{ (1 + q^{k-1/2}) ^2}{(1-q^k) ^2},\quad\text{if }p\text{ is even.}\nonumber
\end{align}
\end{theorem}

Our proof is based on the following steps.  

\begin{itemize}
\item $L[p,r]$ is realised as a submodule of $\mathcal F_{p,r}$ for $p <0$ (cf.\ Proposition \ref{ired-pmanji}).
\item By evaluating  a spanning set of $L[p,r]$ in $\mathcal F_{p,r}$, we get the  inequalities:
\begin{align}
\ch L[p,r] &\le  q^{h_{p,r}} (1 -q ^{\tfrac{\vert p\vert}{2}})   \prod_{k=1} ^{\infty} \frac{ (1 + q^{k-1/2}) ^2}{(1-q^k) ^2},\quad\text{if }p\text{ is odd;}\nn
\ch L[p,r] &\le q^{h_{p,r}} (1 -q ^{\vert p\vert   })   \prod_{k=1} ^{\infty} \frac{ (1 + q^{k-1/2}) ^2}{(1-q^k) ^2},\quad\text{if }p\text{ is even.}\nonumber
\end{align}
\item The opposite inequalities follow by using formulas for singular/subsin\-gu\-lar vectors in $\mathcal F_{p,r}$.
\item The formula for $p >0$ follows by applying the contragradient modules, which must have the same characters.
\end{itemize}

\subsection*{The structure of Verma modules.}  A companion problem is the complete description of the structure of Verma modules and Fock spaces. Although we are able to describe the characters of irreducible modules, our methods cannot solve the question of the structure of these indecomposable modules in full generality.  As it is common in Lie super-algebra theory, in general, one does not have embeddings of Verma modules. However, we prove/conjecture the following:
\begin{itemize}
\item The structure of the Verma modules $V[p,r]$ for $p$ even, is completely determined. In particular, every embedding $V[p, r'] \hookrightarrow V[p,r]$ is injective, and the maximal submodule of $V[p,r]$ is generated by a singular vector (cf.\ Theorem \ref{max-p-even}) for which we present explicit formulas.
\item If $p>0$ is odd there exists a non-injective homomorphism $V[p, r'] \rightarrow V[p,r]$. Moreover, the maximal submodule of $V[p,r]$ is generated by a subsingular vector (\ref{subsing})  (cf.\ Theorem \ref{max-p-odd-manji}).
\item If $p<0$ is odd, we conjecture that the all homomorphisms $V[p, r'] \rightarrow V[p,r]$ are embeddings, but  we cannot prove this.
\item All embedding diagrams (partially conjectured for $p$ odd) are presented in the Appendix.
\end{itemize}

\subsection*{Future work} In a sequel \cite{AR3}, we shall focus on a connection with other logarithmic vertex algebras. Among other results, we  will  prove the following characterisation of 
$V^\SH(c_L,\cla)$:
\begin{theorem}\cite{AR3} We have:
 \bea  V^\SH(c_L,\cla) &=& \mbox{Ker}_{\mathcal F_{-1,0}  } Q\bigcap  \mbox{Ker}_{\mathcal F_{-1,0}}  \mathcal G. \nonumber \eea
\end{theorem}
\vskip 5mm

We would like to thank Maria Gorelik for useful comments related to zero-divisors in Lie superalgebras and embeddings of Verma modules.

D.A. and G.R. are partially supported by the QuantiXLie Centre of Excellence, a project cof{}f{}inanced by the Croatian Government and European Union through the European Regional Development Fund - the Competitiveness and Cohesion Operational Programme (KK.01.1.1.01.0004).

\section{Preliminaries}

\subsection{Partitions and super-partitions}

Recall that a partition in a set $S \subset {\Q}_{\ge 0}$ is a f{}inite sequence  $   \mu = (\mu_1, \mu_2, \dots, \mu_{\ell}) \in S^{\ell}$ of length $\ell = \ell_{\mu} \in \N$ satisfying
\begin{equation}
	\mu_1 \ge \mu _2 \ge \dots \ge \mu_{\ell}.
\end{equation}
The weight of the partition $\mu$ is def{}ined to be $\deg_{\mu} = \mu_1 + \mu_2 + \dots + \mu_{\ell}$.  Let $\mathcal P $ denote the set of all partitions in ${\Z}_{> 0}$.
 
A super-partition in the set  $S \subset {\Q}_{\ge 0}$ is a f{}inite sequence $$\l = (\l_1, \l _2, \dots, \l _{\ell}) \in S^{\ell}$$ of length $\ell  = \ell_{\l}  \in \N$ satisfying
\begin{equation}
	\mu_1  >  \mu _2  >  \dots  > \mu_{\ell}.
\end{equation}
The weight of the superpartition $\l$ is def{}ined to be $\deg_{\l}  = \l_1 + \l_2 + \dots + \l_{\ell}$.  Let $\mathcal {SP} $ denote the set of all super-partitions in $\tfrac{1}{2} + {\Z}_{\ge 0}$.

Def{}ine the partial ordering in $\mathcal P$ (resp.\ $\mathcal{SP}$) by 
 $$ \mu  \prec  \mu'    \quad  \mbox{if} \quad \mu_1 = \mu' _1,  \cdots, \mu_{r-1} = \mu' _{r-1}, \  \mu_{r}  > \mu' _{r}. $$

Now we consider the set $\mathcal P \times \mathcal{SP}$.  For $(\mu, \l) \in \mathcal P \times \mathcal{SP}$, we def{}ine:
 $$ \deg_{\mu, \l} := \deg_{\mu}  + \deg_{\l}, \  \ell_{\mu, \l}:= \ell_{\mu}+ \ell_{\l}.  $$

  For $(\mu^i, \l^i)  \in \mathcal P \times \mathcal{SP}$ $i=1,2$,  we def{}ine the following partial ordering:
\bea
 (\mu^1, \l^1)  <   (\mu^2, \l^2)  &  
\mbox{if} &\deg_{\mu^1, \l^1}   < \deg_{\mu^2, \l^2}  \nonumber \\
& \mbox{or}&\deg_{\mu^1, \l^1}   = \deg_{\mu^2, \l^2},  \  \ell_{\mu^1, \l^1}  < \ell_{\mu^2, \l^2},    \nonumber \\
& \mbox{or}&\deg_{\mu^1, \l^1}   = \deg_{\mu^2, \l^2},\ \ell_{\mu^1, \l^1}  = \ell_{\mu^2, \l^2},\  \mu_1 \prec \mu_2   \nonumber \\
& \mbox{or}&\deg_{\mu^1, \l^1}   = \deg_{\mu^2, \l^2},\ \ell_{\mu^1, \l^1}  = \ell_{\mu^2, \l^2},\  \mu_1 = \mu_2,\ \l_1 \prec \l_2.
 \label{ordering-0} \eea
For an even element $X$ (resp.\ odd element $Y$) and a partition $\mu \in \mathcal P$ (resp.\ a super-partition $\l$), we def{}ine monomials
\bea
&&  X_{-\mu} := X(-\mu_1) \cdots X(-\mu_n),  \ \ X_{\mu} := X(\mu_n) \cdots X(\mu_1), \nonumber  \\ && Y_{-\l}:= Y(-\l_1) \cdots Y(-\l_n),\ \ Y_{\l}:= Y(\l_n) \cdots Y(\l_1). \ \  \nonumber \eea

\subsection{The vertex algebra  $V^\SH(c_L, \cla)$ }  The notion of $N=1$ Heisenberg-Virasoro algebra was introduced in \cite{AJR}. We recall main def{}initions and constructions.

\begin{definition}\label{def}
The $N=1$ Heisenberg-Virasoro algebra $\mathcal{SH}$ is an inf{}inite dimensional Lie algebra with even generators  $L(n)$, $\a(n)$, odd generators $G(n+\frac{1}{2})$, $\ps(n+\frac{1}{2})$, $n\in\Z$, and three central elements $C_{L}$, $C_{\a}$, $\Cla$, subject to the following super-commutator relations:
\bea
&&[\a(m),\a(n)] =\d_{m+n,0}mC_{\a} \nn
&&[L(m),\a(n)] =-n\a(m+n)-\d_{m+n,0}(m^{2}+m)\Cla \nn
&&[L(m),L(n)]= (m-n)L(m+n)+\d_{m+n,0} \frac{m^{3}-m}{12}C_{L} \nn
&&[ \ps (m + \frac{1}{2}), \ps(n+ \ds\frac{1}{2} ) ]_{+}=\d_{m+n+1,0}C_{\a} \nn
&&[\a(m),\ps(n+ \frac{1}{2})]=0 \nn
&&[G (m+ \frac{1}{2}),G (n+  \frac{1}{2}) ]_{+}=2L(m+n+1)+\d_{m+n+1,0} \frac{m^{2}+m}{3}C_{L} \nn
&&[L(m),G (n+  \frac{1}{2})]=( \frac{m}{2}-n- \frac{1}{2})G (m+n+ \frac{1}{2}) \nn
&&[\a(m),G (n+\ \frac{1}{2})]=m\ps(m+n+ \frac{1}{2} ) \nn
&&[\ps(m+\ds\frac{1}{2}),L(n) ]=\frac{2m+n+1}{2}\ps (m+n+ \frac{1}{2}); \nn
&&[\ps (m+\frac{1}{2} ),G (n+ \frac{1}{2}) ]_{+}=\a(m+n+1) + 2m \d_{m+n+1,0} \Cla\nn
&&[\mathcal{SH},C_{\a}]=[\mathcal{SH},C_{L}]=[\mathcal{SH},\Cla]=0\nonumber
\eea 
\end{definition}
Lie superalgebra $\SH$ has the following triangular decomposition:
\bea
\SH &=&  \SH^{-} \oplus   \SH^{0}  \oplus  \SH^{+}, \quad \mbox{where} \nonumber \\
\SH^{\pm} &=& \mbox{span}_{\C} \{ L(\pm n), \alpha (\pm n),  G( \pm (n-1/2) ), \Psi( \pm (n-1/2) ) \vert \ n \in {\Z}_{>0}  \} \nonumber \\
\SH^{0} &=& \mbox{span}_{\C} \{ L(0), \alpha(0),  C_{L}, C_{\a}, \Cla \}. \nonumber 
  \eea

Let $V (c_L,c_\a,\cla,h, h_{\a})$ denote the Verma module of highest weight $(h,h_\a)$ and central charge $(c_L,c_\a,\cla)$. We showed  (\cite{AJR}) that there is a universal vertex algebra associated to  $\mathcal{SH}$, and it is realised as
\bea 
V^\SH (c_L, c_{\a}, \cla)\cong\frac{V (c_L,c_\a,\cla,0,0)}{\<G(-\frac{1}{2})v\>}. \label{voa-as-kvocijent}
\eea
Basis  of $V^\SH (c_L, c_{\a}, \cla)$ consists of monomials
$$ (\Psi_{-\l^-} \a_{-\mu^-} G_{-\l^+} L_{-\mu^+}) {\bf  1}, \quad \mu^+ _i \ne 1, \l ^+ \ne 1/2.$$

The vertex operator $Y(\cdot,z)$ is uniquely determined by
\bea
&&Y(\a(-1) {\bf 1}, z ) = \a(z) = \sum_{n\in\Z}\a(n) z^{-n-1}, \nn
&&Y(L(-2) {\bf 1}, z) = L(z) = \sum_{n\in\Z}L(n)z^{-n-2},  \nn
&&Y \left(\Psi(-\frac{1}{2}) {\bf 1}, z\right) = \ps(z) = \sum_{n\in\Z}\ps(n+\frac{1}{2})z^{-n-1}, \nn
&&Y\left(G(-\frac{3}{2}) {\bf 1}, z\right) = G(z) = \sum_{n\in\Z}G(n+\frac{1}{2})z^{-n-2}. \nonumber
\eea

By construction of the universal vertex algebra, the basis of 
$V^\SH (c_L, c_{\a}, \cla)$ consists of monomials
$$ (\Psi_{-\l^-} \a_{-\mu^-} G_{-\l^+} L_{-\mu^+}) {\bf  1}$$
where $\l^-$  is a super-partition in $\tfrac{1}{2} + {\Z}_{>0}$, $\l^+$ is a super-partition in $\tfrac{3}{2} + {\N}$, $\mu^-$ is a partition in ${\N}$ and  $\mu^+$ a partition in ${\Z}_{>1}$.
The $q$-character of $V^\SH (c_L,c_\a,\cla)$ is thus given by
\bea
  \ch  V^\SH (c_L, c_{\a}, \cla) & = &    \prod_{k=1} ^{\infty} \prod_{l=2} ^{\infty}   \frac{ (1 + q^{k-1/2}) }{(1-q^k)  }  \frac{ (1 + q^{l -1/2}) }{(1-q^l)  } \nonumber \\
&= & (1 -q ^{ 1/2  } ) \prod_{k=1} ^{\infty} \frac{ (1 + q^{k-1/2}) ^2}{(1-q^k) ^2}.      \label{char-universal}
\eea

In \cite{AJR} we showed that when $c_\a \ne 0$,  $V^\SH (c_L,c_\a,\cla)$ is isomorphic to the tensor product of the  $N=1$ Neveu-Schwarz vertex algebra and Heisenberg-Clif{}ford vertex algebra $SM(1)$.
In this paper we study the structure of  the vertex algebra $V^\SH(c_L, \cla):=V^\SH(c_L, 0, \cla)$ and its representations.
For the rest of the paper we assume that $\cla\neq0$ and write $V (h, h_{\a})$ for the Verma module $V (c_L,0,\cla,h, h_{\a})$. Let $L(h, h_{\a})$ denote its simple quotient.

Denote by $V[p,r]$ (resp.\ $L[p,r]$) the Verma module $V(h, h_{\alpha})$ (resp.\ the irreducible highest weight module $L(h, h_{\alpha})$) with the highest weight
\bea (h, h_{\alpha}) : = (h_{p,r}, (1+p) c_{L, \alpha} ), \quad h_{p,r} = (1-p^2)\frac{c_L-3}{24} - rp \label{hpr}  \eea
Note that $h_{p,r}+p=h_{p,r-1}$.
\begin{proposition}\label{realizacija1}
\item[(1)] Let $(h,h_\a)\in\C^2$ such that $h_\a\neq\cla$. Then there exist unique $p,r\in\C$, $p\neq0$ such that $h=h_{p,r}$ and $h_\a=(1+p)\cla$.
\item[(2)] For every $r\in\C$, $h_{0,r}=\frac{c_L-3}{24}$.
\end{proposition}

\vspace{5pt}
In \cite{AJR} we obtained the determinant formula
$$\det[V(h, h_\a)](\cdot, \cdot)  _{n/2} = Const \prod_{  \tiny \begin{array}{ccc}
 	k, l \in {\Z} _{\ge 0} \\ k l \le n \\ k \equiv l \ \mbox{mod} 2
 	\end{array} } \varphi _{k,l} (c_L, c_\a, \cla, h, h_\a)   ^ {p_2 (\frac{n -kl}{2} )}$$
where $\varphi _{k,l} (c_L, \cla, c_\a, h, h_\a) $ denotes
$$\frac{\cla^4 }{4}  \left(1  + k-  \frac{ h_\a} {\cla} \right) \left(-1  + k+  \frac{ h_\a} {\cla} \right) \left(1+ l-  \frac{ h_\a} {\cla} \right) \left(-1+ l +  \frac{ h_\a} {\cla} \right)$$
and $p_2(n)$ is Kostant partition function in $\tfrac{1}{2}\Zp$. As a direct application we get the following:
\begin{theorem}\cite{AJR}\label{struktura}
\begin{enumerate}
    \item The Verma module $V[p,r]$ is irreducible if and only if $\vert p\vert\notin \N$.
    \item If $p\in\Z\setminus\{0\}$ is even, $V[p,r]$ contains a singular vector at conformal weight $\vert p\vert$.
    \item If $p\in\Z$ is odd, $V[p,r]$ contains a singular vector at conformal weight $\vert p/2\vert$.
\end{enumerate}
\end{theorem}

\subsection{Contragredient modules}

As in \cite{AR1} we use the concept of contragredient modules. 

Let $V$ be a vertex operator superalgebra, 
$(M,Y_{M})$ a graded $V$-module with gradation $M=\oplus_{n \in \tfrac {1}{2} {\Z}_{\ge 0} }M(n)$ such that $\dim M(n)<\infty$ and let $\gamma\in{\C}$ such that $L(0)|M(n)\equiv(\gamma+n)\operatorname{Id}$. The contragredient module $M^{\ast}$ is def{}ined as follows. For every $n\in{ \tfrac{1}{2} \Zp}$ let $M(n)^{\ast}$ be the dual vector space and $M^{\ast} = \oplus_{ n \in \tfrac {1}{2} {\Z}_{\ge 0} }M(n)^{\ast}$. Consider the natural pairing $\<\cdot,\cdot\>:M^{\ast}\otimes M\rightarrow\C$. Def{}ine the linear map $Y_{M^{\ast}}:V\rightarrow\operatorname{End}M^{\ast}[[z,z^{-1}]$ such that
\[
\< Y_{M^{\ast}}(v,z)w^{\prime},w\>=\< w^{\prime},Y_{M}(e^{zL(1)} e^{\pi i L(0)}  z^{-2 L(0) } v,z^{-1})w\>
\]
for each $v\in V$, $w\in M$, $w^{\prime}\in M^{\ast}$. Then $(M^{\ast},Y_{M^{\ast}})$
carries the structure of a $V$--module. Direct calculations show that
\begin{eqnarray}
\<L(n)w',w\>&=&\<w',L(-n)w\>\label{contra1}\\
\<\a(n)w',w\>&=&\<w',(-\a(-n)+2\cla\delta_{n,0})w\>\\
 \<G(n+1/2)w',w\>&=&\<w',-\iu G(-n-1/2)w\> \\
\<\Psi(n+1/2)w',w\>&=&\<w',\iu\Psi(-n-1/2)w\>\label{contra2}
\end{eqnarray}

If we take $M$ to be the simple, highest weight module $L(h, h_{\alpha})$, then one gets (cf.\ \cite{FHL}) that $L(h,h_\a)^{\ast}$ is again a simple module and the above calculations shows:
\begin{lemma}
We have
\begin{enumerate}
	\item$L(h,h_\a)^{\ast}\cong L(h,-h_\a+2\cla)$, i.e\ $L[p,r]^{\ast}\cong L[-p,-r]$.
	\item $L(h,h_\a)^{\ast}\cong L(h,h_\a)$ if and only if $h_{\alpha} = \cla$ i.e. $p=0$.
\end{enumerate}
\end{lemma}
From the previous Lemma we have
$$\ch L(h,h_\a)^{\ast} = \ch L(h,h_\a)\qquad\text{i.e.}\qquad\ch L[p,r] = \ch L[-p,-r].$$

\section{Free f{}ield realisation}\label{realizacija}

In this section we recall the realisation from \cite{AJR}. Let $L={\Z}c  + {\Z} d$ be a lattice such that
$\< c,  c \>  = \< d, d\> = 0$, $\< c, d \> =2$, let $V_L = {\C}[L] \otimes M(1)$ be the corresponding lattice vertex algebra, where $M(1)$ is Heisenberg vertex algebra generated by $c(z)$ and $d(z)$.

Consider the vertex subalgebra $\Pi(0) = {\C}[\Z c] \otimes M(1)$ of $V_L$. As in \cite{A-2019}, let $\Pi(0) ^{1/2}$ be its simple current extension: 
$$ \Pi(0) ^{1/2} = \Pi(0) \oplus \Pi(0). e^{\frac{c}{2}. }$$ 
Let $F^{(2)}$ be the fermionic vertex algebra generated by f{}ields
$$\Psi^\pm (z) = \sum _{n \in {\Z}} \Psi^\pm \left(n+ \frac{1}{2}\right) z^{-n-1}$$ 
such that for $i =1,2$, $r,s \in \frac{1}{2} + {\Z}$ we have the following anti-commutator relation
$$ \{ \Psi^\pm (r), \Psi^\pm (s) \} = 0 , \ \{\Psi^+ (r), \Psi^- (s) \} = \delta_{r+s,0}.$$
The vertex algebra $F^{(2)}$ has the following Virasoro vector of central charge $c_{fer} =1 $:
$$\omega_{fer}= \frac{1}{2} \left( \Psi^+(-\frac{3}{2}) \Psi^- (-\frac{1}{2}) +  \Psi^-(-\frac{3}{2}) \Psi^+ (-\frac{1}{2}) \right) {\bf 1}. $$
Def{}ine the following four vectors in the vertex algebra $M(1) \otimes F^{(2)}$:
\bea
\alpha &=& - \cla c(-1) \nonumber \\
\tau &=&\sqrt{2}\left(\frac{1}{2} c(-1) \Psi^+(-\frac{1}{2}) + \frac{1}{2} d(-1) \Psi^- (-\frac{1}{2}) + \frac{c_L - 3}{12} \Psi^-(-\frac{3}{2}) - \Psi^+(-\frac{3}{2})\right) \nonumber \\
\omega &=& \frac{1}{2} c(-1) d(-1) + \frac{c_L -3}{24} c(-2) - \frac{1}{2} d(-2) + \omega_{fer} \nonumber \\
\Psi &=& -\sqrt{2} \cla\Psi^- (-\frac{1}{2}). \nonumber 
\eea

\begin{theorem} \label{simplicity-voa}
The universal vertex algebra  $V^\SH(c_L, \cla)$ is simple, and it is isomorphic to 
the vertex subalgebra of $M(1) \otimes F^{(2)}$ generated by $\alpha, \Psi, \tau, \omega$.  
\end{theorem}
\begin{proof}
We have shown in \cite{AJR} that the vertex algebra  $W$ generated  by $\alpha, \Psi, \tau, \omega$ is isomorphic to some quotient of $V^\SH(c_L,\cla)$.   
Note  that $V^\SH(c_L,\cla)$ is isomorphic to a certain quotient of  the Verma module $V[-1,0]$. But we will prove in Theorem \ref{xx}, that the character of  $V^\SH(c_L,\cla)$, given by (\ref{char-universal}), coincides with the character of $L[-1,0]$. This proves that  $V^\SH(c_L,\cla)$ is simple. In particular,  $W= V^\SH(c_L,\cla)$.
\end{proof}

Now for each $h\in\C\otimes_\Z L$, let $e^h$ denote a $M(1)$ highest weight vector in the Fock module $M(1,h),$ which is also a $V^\SH(c_L,\cla)$-module. We introduce a parametrisation 
\begin{equation}\label{param}
\vpr:=e^{-\frac{p+1}{2}\overline d+r c},\qquad\text{where}\quad\overline d=d-\tfrac{c_L-3}{12}c.
\end{equation}
Then $\vpr$ is the highest weight vector of highest weight given by (\ref{hpr}).
Also, def{}ine $$\mathcal F_{p,r} =  ( M(1) \otimes F^{(2)}) . \vpr. $$
Now we shall identify  the contragredient module $\mathcal F_{p,r} ^{\ast}$. Using
(\ref{contra1}-\ref{contra2}) and
\begin{eqnarray*}
 \<c(n)w',w\>&=&\<w',c(-n)+2\delta_{n,0}w\>,\\
\<d(n)w',w\>&=&\<w',d(-n)-\delta_{n,0}\frac{c_L-3}{6}w\>,\\
\<\Psi^{\pm} (n+1/2)w',w\>&=&\<w',\iu\Psi^{\pm} (-n-1/2)w\>,
 \end{eqnarray*}
we get 
$$ \mathcal F_{p,r}^{\ast}\cong\mathcal F_{-p,-r}. $$

\subsection{Realisation of Verma modules $V[p,r]$,  $p \in {\C} \setminus {\Z}_{<0}$.}\label{racun}

In this section, we shall prove that Verma modules $V[p,r]$, $p \notin {\Z}_{<0}$ can be obtained by using free-f{}ield realisation.

Let $ \mathcal W = V^\SH(c_L,\cla).\vpr$. Take the following basis of $\mathcal F_{p,r} $ consisting of monomials:
\bea
   w_{\l^+, \l^-, \mu^+, \mu ^-} &=&( \Psi^{+}  _{-\l^+} \Psi ^-  _{- \l^-} d_{-\mu^+ } c_{-\mu^- } )  \vpr  \label{PBW-2} 
\eea
where $\l^{\pm} \in \mathcal{SP}$, and $\mu^{\pm}  \in \mathcal P$.
Now, since $[c(n), c(m)]=[d(n), d(m)]=0$, $\{ \Psi^\pm (r), \Psi^\pm (s) \} = 0$, and $\{\Psi^+ (r), \Psi^- (s) \} = \delta_{r+s,0}$ and from the def{}inition of $v_{p, r}$, one can def{}ine the following partial ordering on basis vectors:
\bea
 w_{\l^+, \l^-, \mu^+, \mu^-} <    w_{\bar \l^+, \bar  \l^-, \bar \mu^+, \bar \mu^-}  \  \mbox{if}  \  (\l^+, \mu^+) <   (\bar \l^+,  \bar \mu^+). \label{ordering-1} \eea
Let us check whether the arbitrary basis vector (\ref{PBW-2}) belongs to $\mathcal W$.

Let $ \deg_{\l^+, \mu^+}   = 0$ (then automatically  $\ell_{\l^+, \mu^+} = 0$).
Since $$ \a_{-\mu}=(-\cla)^{\ell_\mu} c_{-\mu},\quad\Psi_{-\l}=(-\sqrt{2}\cla)^{\ell_\l}\Psi^-_{-\l},$$
we have 
$$( \Psi^-_{-\l^-}  c_{-\mu^-} )  \vpr  \in \mathcal W.$$
Let $S$ be the set of all basis vectors which don't belong to  $\mathcal W$. Assume that $S\ne \emptyset$. By the Zorn's lemma there must exist a minimal element $w:=w_{\l^+, \l^-, \mu^+, \mu^-}$ of $S$ with respect to the ordering "$<$". This means that for every $w' <w$, $w' \in \mathcal W$.

Assume f{}irst that $\ell_{\mu^+} = l >0$. Let $\mu^+= (\mu_1, \mu_2, \dots, \mu_l)$ and 
$\bar \mu^+ :=  (\mu_2, \dots, \mu_{l})$. Def{}ine 
$$ w' = ( \Psi^{+} _{-\l^+} \Psi ^- _{-\l^-}    d_{-\bar \mu^+}  c _{-\mu^-}  ) \vpr $$
Then by the assumption $w' \in \mathcal W$. We have
\bea
&&L(-\mu_1) w' =- \frac{p+\mu_1}{2} w+\cdots \label{ff_L}
\eea
where $\cdots$ denotes a sum of monomials $w_i$ such that  $w_i  <  w$.
Using  again the assumption, we have that all $w_i$ and $L(-\mu_1) w'$ belong to $\mathcal W$, so if $p+\mu_1\neq0$ we have $w \in \mathcal W$.

Assume next that $\ell_{\mu^+}  = 0$.  
Let $\l^+ = (\l_1 ,  \l_2, \dots, \l_g)$  and $\bar \l^+ = (\l_2 , \dots, \l_{g})$.
Def{}ine
$$w' = (\Psi^{+} _{-\bar \l^+}  \Psi ^- _{-\l^-}  c_{-\mu^-} )  \vpr$$
Then we get
\bea
&&G (-\l_1)  w'=  - \frac{\sqrt{2}}{2} (2 {\l_1} +p) w   + \cdots \label{ff_G}
\eea
where $\cdots$ denotes a sum of monomials $w_i$ such that  $w_i < w$. As before, we conclude that if $2 {\l_1} +p\neq0$ then $w \in \mathcal W$.

\begin{proposition}\label{realizacija_Verma}
Assume that $p \in {\C} \setminus {\Z}_{<0}$.  Then we have:
$$\mathcal F_{p,r} \cong V[p,r].$$
\end{proposition}
\begin{proof}
We have shown above that all basis vectors (\ref{PBW-2}) belong to $\mathcal W$ if $p\notin\Z_{<0}$. Since the $q$--character of Verma module $V[p,r]$ coincides with $q$--character of the Fock space, we conclude
$\mathcal W = V[p,r] = \mathcal F_{p,r}$.
\end{proof}

\subsection{Realisation of $L[p,r]$,  $p \in \Z_{<0}$ as submodules of $\mathcal F_{p,r}$}

Let $\mathcal C_{p,r} $ be the subspace of $\mathcal F_{p,r}$ spanned by monomials $\Psi_{-\l}\a_{-\mu}\vpr$, $\l\in\mathcal{SP}$, $\mu\in\mathcal P$. 
We f{}irst show that there can be no singular vectors in $\mathcal C_{p,r}$, and then generalise the claim to $\mathcal F_{p,r}$.

Since we are dealing with (anti)commuting factors, we will consider $(\mu,\l)\in\mathcal P\times\mathcal{SP}$ as an element of a partially ordered set of partitions  in $\tfrac{1}{2}\N$. Let $\mathcal P(\tfrac{1}{2})$ denote the set of all superpartitions $(\l_1,\ldots,\l_n)$ in $\tfrac{1}{2}\N$ such that
$$\l_i\ne\l_{i+1}\quad\text{if}\quad\l_i\in\tfrac{1}{2}+\Zp.$$
We introduce a partial ordering $<$ on $\mathcal P(\tfrac{1}{2})$ by
\bea
\l < \mu & \text{if}& \deg_\l<\deg_\mu,\nn
	&\text{or}& \deg_\l=\deg_\mu,\ \l\prec\mu.\label{ordering-1}
\eea
For each $\l\in\mathcal P(\tfrac{1}{2}),$ we def{}ine monomials
\[
X^{\a,\Psi}_{-\l}= X^{\a,\Psi}(-\l_1)\cdots X^{\a,\Psi}(-\l_{\ell(\l)}),\qquad Y^{L,G}_\l=Y^{L,G}(\l_{\ell(\l)})\cdots Y^{L,G}(\l_1)
\]
where
\bea
X^{\a,\Psi}(-\l_i)=\begin{cases}\a(-\l_i)&\l_i\in\N,\\
				\Psi(-\l_i)&\l_i\in\tfrac{1}{2}+\Zp
	\end{cases}&Y^{L,G}(\l_i)=\begin{cases}
L(\l_i)&\l_i\in\N\\
G(\l_i)&\l_i\in\tfrac{1}{2}+\Zp.
\end{cases}\nonumber
\eea

A basis of $\mathcal C_{p,r}$ is then given by the set of monomials 
\bea
w_{\l}=X^{\a,\Psi}_{-\l} \vpr,\qquad\l\in\mathcal P(\tfrac{1}{2}).\label{monom-1}
\eea

\begin{lemma}\label{annih}
Let $p\in\Z_{<0}$. Let $\l,\mu\in\mathcal P(\tfrac{1}{2}\N)$ such that $\mu<\l$. Then $Y^{L,G}_\l X^{\a,\Psi}_{-\mu}\vpr=0$ and $Y^{L,G}_\l X^{\a,\Psi}_{-\l}\vpr=\nu\vpr$ for some $\nu\ne0$.
\end{lemma}
\begin{proof}
Both claims follow from the def{}inition of ordering (\ref{ordering-1}). If $\deg_\mu<\deg_\l$ we obviously have $Y^{L,G}_\l X^{\a,\Psi}_{-\mu}\vpr=0$. By the brackets in Def{}inition \ref{def} we have for $\l_k\geq\mu_{1}$
$$Y^{L,G}(\l_k)X^{\a,\Psi}_{-\mu}\vpr=\nu_k\partial_{\l_k}X^{\a,\Psi}_{-\mu}\vpr$$
where
$$\nu_k=\begin{cases}
\mu_1(p-\mu_1),&\mu_1\in\N,\\
p-\mu_1-1,&\mu_1\in\tfrac{1}{2}+\N
\end{cases}.$$
From this follows that $Y^{L,G}_\l X^{\a,\Psi}_{-\mu}\vpr=0$, and $Y^{L,G}_\l X^{\a,\Psi}_{-\l}\vpr=\nu\vpr$ where
$$\nu = \prod(p-\mu_i-1)\prod\l_i(p-\l_i).$$
$\nu\neq0$ since $p\in\Z_{<0}$.
\end{proof}
\begin{lemma}  \label{nema-sing-c}
If $p\in\Z_{<0}$ the subspace  $\mathcal C_{p,r} $ of $\mathcal F_{p,r}$ (resp.\ of $V[p,r]$) contains no singular vectors.
\end{lemma}
\begin{proof}
Assume that 
$$ u=\sum_{ \l\in \mathcal P(\tfrac{1}{2})}  k_{\l} X^{\a,\Psi}_{-\l} \vpr, $$
is a singular vector of conformal weight $h$ (so $\deg_\l=h$ for each $\l$).
Let $$\bar\l = \min \lbrace\l:k_{\l}\neq0\rbrace$$ with respect to linear ordering (\ref{ordering-1}) on the set of partitions in $\tfrac{1}{2}\Z$ of total degree $h$. Consider $Y^{L,G}_{\bar\l} u$. By Lemma \ref{annih} we have 
\bea
Y^{L,G}_{\bar\l} X^{\a,\Psi}_{-\bar\l}&=&\nu\vpr,\nn
Y^{L,G}_{\bar\l} X^{\a,\Psi}_{-\l}&=&0\text{,\qquad for }\l>\bar\l,\nonumber
\eea
for some $\nu\neq0$. Therefore we have $Y^{L,G}_{\bar\l}u=\nu\vpr\ne0$. This is a contradiction with the assumption that $u$ is a singular vector.
\end{proof}

\begin{proposition} \label{ired-pmanji}
For each non-zero vector $x \in  \mathcal F_{p,r}$, the cyclic $V^{\SH}(c_L,\cla)$-submodule $\langle x \rangle$ contains $v_{p,r}$. In particular, $ L[p,r] = \langle v_{p,r} \rangle$ if $p\in\Z_{<0}$.
\end{proposition}
\begin{proof}  
Let $U =\langle x \rangle$. Taking the actions of $$\alpha(n)= -2 n \cla   \partial_{d(-n)} ,  \quad \Psi(n-1/2) = -\sqrt{2}  \cla \partial_{\Psi^+ (-n-1/2)} $$ for $n \in {\Z}_{>0}$, we  see that $U$ must contain some element
$w \in \mathcal C_{p,r}$. Using Lemma \ref{nema-sing-c} we get that $v_{p,r} \in \langle w \rangle$. Applying this result we have that each vector in  $\langle v_{p,r} \rangle$ is cyclic. Hence  $ L[p,r] = \langle v_{p,r} \rangle$.
\end{proof}

\section{Screening operators and singular vectors, case $p>0$}\label{screening}

\subsection{Screening operators} Let $a = \Psi^-(-\frac{1}{2}) e^{\frac{1}{2} c}$. We have
\bea
\tau_0 a &=& \sqrt{2} D e^{\frac{1}{2} c}, \nonumber\\
\tau_1 a &=& \sqrt{2}  e^{\frac{1}{2} c}  \nonumber \\
\tau_n a & = & 0 \qquad (n \ge 2) \nonumber \\
L(0) a &=& a \nonumber \\
L(m) a &=& 0  \qquad (m \ge 1) \nonumber 
\eea
By using the commutator formula, we get that $Q= a_0 = \mbox{Res}_z Y(a,z)$ either commutes or anti-commutes with the generators of  $V^\SH(c_L, \cla)$. Hence $Q$ is a screening operator.

We also have
$$ a_1 \tau = [ \tau_{-1}, a_1] {\bf 1} = (\tau_0 a )_0 {\bf 1} - (\tau_1 a)_{-1} {\bf 1} =  - \sqrt{2}  e^{\frac{1}{2} c}. $$
$$ a_n \tau = 0 \quad \forall n \ge 2. $$

Next we notice that
\bea
\{ a_n, a_m \} = 0, \quad \forall n, m \in {\Z}\nonumber,
\eea
hence $a$ is an odd vector with anti-commuting components. 
In  \cite{AdM-2010}, we studied a construction of a derivation based on the element $a$.

Def{}ine
\begin{eqnarray*}
S &=& \sum_{i =1 } ^{\infty} \frac{1}{i} a_{-i} a_{i} = \mbox{Res}_{z_1}   \mbox{Res}_{z_2}  \mbox{Log} (1 - \frac{z_2}{z_1}) Y(a, z_1) Y (a, z_2).\\
S^{tw} &=& \sum_{i =0 } ^{\infty} \frac{1}{i+1/2} a_{-i-1/2} a_{i+1/2}.
\end{eqnarray*}
We have:
\begin{itemize}
 \item $S$ is an even derivation on $ \overline V = \Ker_{\Pi(0) ^{1/2} \otimes F^{(2)}} Q $. 
 \item On every $\Pi(0) ^{1/2} \otimes F^{(2)}$--module $(M,  Y_M)$   and $ v \in  \overline V$,  we have $$ [ S, Y_M (v,z)] = Y_M (S v, z).$$
  \item On every  $\sigma$--twisted $\Pi(0) ^{1/2} \otimes F^{(2)}$--module $(M^{tw},  Y_{M^{tw} })$   and $ v \in  \overline V$,  we have $$ [S^{tw}, Y_{M^{tw} } (v,z)] = Y_{M^{tw}} (S v, z).$$

\item In particular, we have
$$ [S, L(n)] = [S , \alpha(n)] = [S, \Psi (n+1/2)  ]= 0 \quad \forall n \in {\Z}. $$
$$ [ S , G(n-1/2)] = [S, \tau_n ] = (S \tau )_n = -    \sqrt{2}  (a_{-1} e^{\frac{1}{2} c} ) _n = - \sqrt{2} (  \Psi^-(-1/2) e^c )_n.  $$

Note also that $e^c_0$ is an (even) derivation. We have:
\begin{eqnarray*}
&&[e^c_0, L(n)] = [e^c_0, \alpha(n)] = [e^c_0, \Psi (n+1/2) ] = 0.\\
&&[e^c _0, G(n-1/2)] = (e^ c _0 \tau )_ n = - \sqrt{2} (\Psi^-(-1/2) e^c )_n.
\end{eqnarray*}
\end{itemize}
From these considerations we conclude:
\begin{theorem}\label{screeninzi}
Let  $$ \mathcal G = e^c_0 - S, \ \mathcal G^{tw} = e^c_0- S^{tw}. $$
\begin{itemize}
\item Assume that  $(M, Y_M)$ is a   $\Pi(0) ^{1/2} \otimes F^{(2)}  $--module. Then on the submodule $\overline M = \Ker_M Q$ we have:
$$[\mathcal G, Y_M (v, z) ] = 0, \quad \forall v \in  V^{\SH} (c_L, \cla). $$
\item Assume that  $(M^{tw}, Y_{M^{tw} } )$ is a $\sigma $--twisted $\Pi(0) ^{1/2}  \otimes F^{(2)} $--module. Then
$$[\mathcal G^{tw}, Y_{M^{tw}} (v, z) ] = 0, \quad \forall v \in  V^{\SH} (c_L, \cla). $$
\end{itemize}
Therefore $\mathcal G$ and $\mathcal G^{tw}$ are screening operators.
\end{theorem}

Def{}ine the Schur polynomial
$S_r(\alpha):= S_r(\alpha(-1), \alpha(-2), \cdots)$ using the generating function
$$ \exp \left(   \sum_{n=1} ^{\infty}\alpha(-n) \frac{z^n}{n}   \right) = \sum_{r=0} ^{\infty} S_r( \alpha) z^r. $$
In particular we have
$$ S_0(\alpha) = 1, \ S_1 (\alpha) = \alpha(-1), \ S_2 (\alpha)= \frac{1}{2} \left( \alpha(-1) ^2 + \alpha(-2) \right) $$
and
\begin{equation}\label{an}
a_n v_{p,r-\tfrac{1}{2}} = -\frac{1}{\sqrt 2 \cla}\sum_{i=0}^{\tfrac{p-1}{2}-n}\Psi(-i-\tfrac{1}{2}) S_{\tfrac{p-1}{2}-n-i}\left(-\tfrac{\a}{2\cla}\right)\vpr.
\end{equation}

\subsection{(Sub)singular vectors for $p \in \N$, $p$ odd}

\begin{theorem}\label{podd}
Assume that $p\in\N$ is odd. Then
\be\label{sing_nep}
u_{p,r}=\sum_{i=0}^{\frac{p-1}{2}}\Psi(-i-\tfrac{1}{2}) S_{\frac{p-1}{2}-i}\left(-\tfrac{\a}{2\cla}\right) \vpr
\ee
is a singular vector (of weight $h_{p,r}+p/2=h_{p,r-1/2}$) in $V[p,r] = \mathcal F_{p,r}$.
\end{theorem}

\begin{proof}
We know from Proposition \ref{realizacija_Verma} that $V^\SH(c_L, \cla).\vpr\cong V[p,r]$. 
By applying the screening operator $Q$ on $v_{p,r - \tfrac{1}{2}} $ we get
$$Q v_{p,r - \frac{1}{2}} = a_0 v_{p,r - \frac{1}{2}}.$$
Since $v_{r - \frac{1}{2}}$ is a highest weight vector, and $Q$ (anti)commutes with the action of $V^\SH(c_L, \cla)$, it is clear that $a_0 v_{p,r - \frac{1}{2}}$ is a singular vector in $V[p,r]$. The formula follows from (\ref{an}).
\end{proof}

\begin{remark} One can show using results from \cite{AdP-2019} or \cite{CGN} that
$\Ker_{\mathcal F_{-1,0}} Q $ is isomorphic to the simple affine vertex superalgebra $L_1(\mathfrak{gl}(1\vert 1))$ associated to the Lie superalgebra $\mathfrak{gl}(1\vert 1)$. We won't use this identif{}ication in the current paper. \end{remark}

Now we show that  there exist subsingular vectors in $V[p,r]$ for $p>0$ odd.

$$\mathcal M_{p,r}:=(\Pi(0) ^{1/2} \otimes F^{(2)}). \vpr$$ is an (untwisted) $\Pi(0) ^{1/2} \otimes F^{(2)}$--module. Let
$$\overline{ \mathcal M_{p,r} } = \Ker_{\mathcal M_{p,r}} Q. $$
Then $\mathcal G$ is a screening operator on $\overline{ \mathcal M_{p,r} }$.

\begin{proposition}\label{screening2}
Let $p \in \N$ be odd. Then
\begin{multline}\label{subsing}
w_{p,r} = \Bigg(S_{p}\left(-\tfrac{\a}{\cla}\right)+ \\
\sum_{k=1}^\frac{p-1}{2} \frac{1}{k}\bigg(\sum_{i\geq0}\Psi(i+k-\tfrac{p}{2})S_i(-\tfrac{\a}{2\cla})\bigg)\bigg(\sum_{j\geq0}\Psi(j-k-\tfrac{p}{2})S_j(-\tfrac{\a}{2\cla})\bigg)\Bigg)\vpr
\end{multline}
is a non-trivial subsingular vector (of weight $h_{p,r}+p=h_{p,r-1}$) in $V[p,r] =\mathcal F_{p,r}$.
\end{proposition}

\begin{proof}Note the following facts:

\begin{itemize}
\item $[Q,  \mathcal G] = 0$, $ Q^2 = 0$.

\item $v_{p,r-1}$ is a singular vector for every $r \in {\C}$.

\item $ Q v_{p,r-1} $ is a singular vector in $\overline{ \mathcal M_{p,r} }$ for every $r \in {\C}$.

\item $  \mathcal G Q v_{p,r-1} $ is a singular vector for every $r \in {\C}$. It is non-trivial since it has the form 
$$ (\Psi^- (-1/2) e^{c/2} )_0 e^c_0 v_{p,r-1} + \cdots   $$
where $\cdots$ denotes the sum of monomials  containing the product of three  fermionic generators  "$\Psi^- (j_1) \Psi^- (j_2) \Psi^- (j_3)$", 
and  $$ (\Psi^- (-1/2) e^{c/2} )_0 e^c_0  v_{p,r-1}  \ne 0. $$

\end{itemize}

Let $w_{p,r} =  \mathcal G v_{p,r-1} $. The arguments above show that  $Q w_{p,r} \ne 0$, but for $n \ge 1$ it holds that
$$ L(n) w_{p,r}, \: \alpha(n) w_{p,r} , \: \Psi\left(n-\tfrac{1}{2}\right) w_{p,r}, \: G\left(n-\tfrac{1}{2}\right) w_{p,r} \in  \overline{ \mathcal M_{p,r} }. $$ 
Hence $w_{p,r}$ is a singular vector in $ \mathcal M_{p,r}  /  \overline{ \mathcal M_{p,r} }$, and therefore subsingular in $ \mathcal M_{p,r}$.
The proof follows.
\end{proof}

\begin{theorem} \label{sing-p-odd-1}
Assume that $p>0$ odd. Then we have the following family of (sub)singular vectors in $V[p,r] = \mathcal F_{p,r}$:
\begin{itemize}
\item Singular vector $u_{p, r} ^{(n)}  = \mathcal G ^n   Q v_{p, r-n-1/2}$, $n \in \Zp$.
\item Subsingular vectors $w_{p, r} ^{(n)}  = \mathcal G ^n v_{p, r-n}$, $n \in \N$.
\end{itemize}
\end{theorem}

Note that
\bea\label{gen}
G\left(\tfrac{p}{2}\right)w_{p,r}^{(1)} &=& [G\left(\tfrac{p}{2}\right),\mathcal G] v_{p,r-1} = \left(\left[G\left(\tfrac{p}{2}\right),e^c_0\right]-\left[G\left(\tfrac{p}{2}\right),S\right]\right) v_{p,r-1} =\nn
	&=& \sum_{i=0}^{\frac{p-1}{2}}\frac{\Psi\left(-i-\tfrac{1}{2}\right)}{\cla} S_{\frac{p-1}{2}-i}(c)\vpr = u_{p,r}^{(0)}.
\eea
From (\ref{subsing}) and (\ref{gen}) it follows that all (sub)singular vectors in Theorem \ref{sing-p-odd-1} belong to a submodule $\<w_{p,r}^{(1)}\>$.

\subsection{Singular vector for $p \in\N$, $p$ even}\label{sing_formula}

Let $p \in 2\N$. As in \cite{A-2019}, we can construct twisted $\Pi(0) ^{1/2}$--modules.

Note that $ \sigma = e^{  \pi i d(0)}$ is an automorphism of the vertex operator algebra $\Pi(0) ^{1/2}$ of order two,  and $\Pi(0) ^{1/2} . e^{-\frac{p+1}{2} \overline d + r c}$, $r \in {\C}$, is a $\sigma$--twisted $\Pi(0) ^{1/2}$--module.
\begin{theorem}\label{singeven}
Assume that $p\in\N$ is even. Then 
\begin{multline}\label{sing_par}
u_{p,r}=\Bigg(S_{p}\left(-\tfrac{\a}{\cla}\right)+ \\
\sum_{k=0}^\frac{p-1}{2} \frac{1}{k+\tfrac{1}{2}}\bigg(\sum_{i\geq0}\Psi(i+k-\tfrac{p-1}{2})S_i(-\tfrac{\a}{2\cla})\bigg)\bigg(\sum_{j\geq0}\Psi(j-k-\tfrac{p+1}{2})S_j(-\tfrac{\a}{2\cla})\bigg)\Bigg)\vpr
\end{multline}
is a singular vector (of weight $h_{p,r}+p=h_{p,r-1}$) in $V[p,r]=\mathcal F_{p,r}$.
\end{theorem}
\begin{proof}
Since $(\Pi(0) ^{1/2} \otimes F^{(2)}) . v_{p,r-1} $ is   $ \sigma = e^{  \pi i d(0)}$--twisted  $\Pi(0) ^{1/2} \otimes F^{(2)}$--module, operator $\mathcal G^{tw}$ gives a screening operator which commutes with the action of the vertex algebra
 $V^\SH(c_L, \cla)$. Therefore
 $\mathcal G^{tw} v_{p,r-1}$ is a singular vector. Direct calculation shows that
\begin{equation*}
\mathcal G^{tw} v_{p,r-1} = e^c _0 v_{p,r-1} - S^{tw} v_{p,r-1} = S_p (c)  \vpr - \sum_{j=0}^{\infty}   \frac{1}{j+\frac{1}{2} } a_{- j -\tfrac{1}{2} }  a_{ j+\tfrac{1}{2}  } v_{p,r-1}.
\end{equation*}
The proof follows. 
\end{proof}

\begin{theorem}  \label{sing-p-even-1}
Assume that $p\in\N$ is even. Then we have the following family of  singular vectors in $V[p,r] = \mathcal F_{p,r}$:
$$ u_{p,r} ^{(n)} := ( \mathcal G^{tw} ) ^n  v_{p, r-n},\ n \in \N.$$
\end{theorem}

Note next that the screening operator $ \mathcal G^{tw}$ is injective, and therefore it gives an inclusion of the Verma module
$V[p,r-1] = \mathcal F_{p,r-1}$ into $V[p,r]$. We have the following conclusion:

\begin{corollary} \label{embed-verma-1} Assume that $p\in\N$ is even. The Verma module $V[p,r]$ contains a non-trivial submodule isomorphic to the Verma module $V[p, r-1]$.
\end{corollary}
Free f{}ield realisation in case $p\in\N$ even is shown in Figure \ref{ff_paran_poz}.

\section{Relations in $V[p,r]$, $p<0$}
We f{}irst present an explicit derivation of a singular vector in $V[p,r]$, $p<0$. Our analysis is analogous to that of \cite[Section 4]{AR2}.

By direct calculations in $\Pi(0)^{(1/2)}\otimes F^{(2)}$ we f{}ind
\begin{align*}
\frac{1}{\sqrt{2}}\Psi^-(-\tfrac{1}{2})G(-\tfrac{3}{2})e^{-c} &= \left(\frac{1}{2}c(-1)\Psi^-(-\tfrac{1}{2})\Psi^+(-\tfrac{1}{2})+\right.\\
	&-\left.\frac{c_L-15}{12}\Psi^-(-\tfrac{3}{2})\Psi^-(-\tfrac{1}{2})+\Psi^+(-\tfrac{3}{2})\Psi^-(-\tfrac{1}{2})\right)e^{-c}\\
L(-2)e^{-c}&=\left(\frac{1}{2}c(-1)d(-1)-\frac{1}{2}d(-2)+\frac{c_L-27}{24}c(-2)+\right.\\
	&+\left.\frac{1}{2}\Psi^+(-\tfrac{3}{2})\Psi^-(-\tfrac{1}{2})+\frac{1}{2}\Psi^-(-\tfrac{3}{2})\Psi^+(-\tfrac{1}{2})\right)e^{-c}
\end{align*}
so we have
\begin{align*}
&\left(L(-2)-\frac{c_L-27}{24}c(-2)-\frac{1}{\sqrt{2}}\Psi^-(-\tfrac{1}{2})G(-\tfrac{3}{2})-\frac{c_L-15}{12}\Psi^-(-\tfrac{3}{2})\Psi^-(-\tfrac{1}{2})\right)e^{-c}=\\
&\qquad=-\frac{1}{2}L(-1)(d(-1)-\Psi^-(-\tfrac{1}{2})\Psi^+(-\tfrac{1}{2}))e^{-c}
\end{align*}

Def{}ine
\begin{align*}
R:=&\left(L(-2)-\frac{c_L-27}{24}c(-2)-\frac{1}{\sqrt{2}}\Psi^-(-\tfrac{1}{2})G(-\tfrac{3}{2})+\right.\\
	&\left.-\frac{c_L-15}{12}\Psi^-(-\tfrac{3}{2})\Psi^-(-\tfrac{1}{2})\right)e^{-c}.
\end{align*}
Since $(L(-1)a)_0=0$ in every VOA, we conclude that $R_0\vpr=0$.

\begin{proposition}\label{sing_neg_p}
Let $p\in\Z_{<0}$, $r\in\C$. Then $u_{p,r}=\Phi(p,r)\vpr$, where
\begin{align*}
\Phi(p,r) =& \sum_{i=1}^{-p}\left(L(-i)+\frac{c_L-27}{24\cla}\a(-i)\right)S_{-p-i}\left(\frac{\a}{\cla}\right)+\\
+&S_{-p}\left(\frac{\a}{\cla}\right)\left(L(0)+\frac{c_L-3}{24\cla}\a(0)\right)+\\
+&\frac{1}{2\cla}\sum_{i=0}^{-p-1}\sum_{k=0}^{-p-i-1}\Psi\left(-i-\frac{1}{2}\right)G\left(-k-\frac{1}{2}\right)S_{-p-i-k-1}\left(\frac{\a}{\cla}\right)+\\
-&\frac{c_L-15}{24\cla^2}\sum_{i=0}^{-p-1}\sum_{k=0}^{-p-i-1}i\Psi\left(-i-\frac{1}{2}\right)\Psi\left(-k-\frac{1}{2}\right)S_{-p-i-k-1}\left(\frac{\a}{\cla}\right).
\end{align*}
is a non-trivial singular vector (of weight $h_{p,r}-p=h_{p,r+1}$) in $V[p,r]$.
\end{proposition}
Unfortunately, this method cannot be used for a construction of singular vectors when $p<0$ is odd.

\section{The $q$--character of $L[p,r]$}

\begin{lemma}\label{nejed}
We have:
\begin{align}
\ch L[p,r] &\ge q^{h_{p,r}} (1 -q ^{\tfrac{\vert p\vert}{2}})   \prod_{k=1} ^{\infty} \frac{ (1 + q^{k-1/2}) ^2}{(1-q^k) ^2},\quad\text{if }p\text{ is odd;}\nn
\ch L[p,r] &\ge q^{h_{p,r}} (1 -q ^{\vert p\vert   })   \prod_{k=1} ^{\infty} \frac{ (1 + q^{k-1/2}) ^2}{(1-q^k) ^2},\quad\text{if }p\text{ is even.}\nonumber
\end{align}
\end{lemma}
\begin{proof}
It suf{}f{}ices to prove the statements for $p<0$. The proof for $p>0$ then follows by using contragredient modules.

We have
$$(G_{-\l^+}\Psi_{-\l^-}L_{-\mu^+}\a_{-\mu^-})\vpr = \nu(\Psi^+_{-\l^+}\Psi^-_{-\l^-}d_{-\mu^+}c_{-\mu^-})\vpr+\cdots$$
where
$$\nu=(-\cla)^{\ell_{\mu^-}}(-\sqrt{2}\cla)^{\ell_{\l^-}}\prod_{i=1}^{\ell_{\mu^+}}\frac{\mu^+_i+p}{-2}\prod_{k=1}^{\ell_{\l^+}}\frac{2\l^+_1+p}{-\sqrt{2}}$$
and $\cdots$ denotes a sum of basis vectors $w_i$ of the type (\ref{PBW-2}) such that $w_i<(\Psi^+_{-\l^+}\Psi^-_{-\l^-}d_{-\mu^+}c_{-\mu^-})$ where we consider the ordering (\ref{ordering-0}) with respect to pair $(\mu^+,\l^+)$. Therefore, the set of vectors
\bea\label{PBW-o}
&&(G_{-\l^+}\Psi_{-\l^-}L_{-\mu^+}\a_{-\mu^-})\vpr,\qquad\mu^+_i\ne-p,\ \l^+_i\ne-p/2
\eea
is linearly independent in $L[p,r]\subset \mathcal F_{p,r}$ if $p$ is odd. Likewise vectors
\bea\label{PBW-e}
&&(G_{-\l^+}\Psi_{-\l^-}L_{-\mu^+}\a_{-\mu^-})\vpr,\qquad\mu^+_i\ne-p
\eea
are linearly independent if $p$ is even.

The $q$--character of the subspace of $L[p,r]$ spanned by vectors (\ref{PBW-o}) is 
$$q^{h_{p,r}} (1 -q ^{\tfrac{\vert p\vert}{2}})   \prod_{k=1} ^{\infty} \frac{ (1 + q^{k-1/2}) ^2}{(1-q^k) ^2}$$
while the $q$--character of the space spanned by (\ref{PBW-e}) is 
$$q^{h_{p,r}} (1 -q ^{\vert p})   \prod_{k=1} ^{\infty} \frac{ (1 + q^{k-1/2}) ^2}{(1-q^k) ^2}$$
which proves both inequallities.
\end{proof}

\begin{lemma}\label{nejed2} We have:
\begin{align}
\ch L[p,r] &\le q^{h_{p,r}} (1 -q ^{\tfrac{\vert p\vert}{2}})   \prod_{k=1} ^{\infty} \frac{ (1 + q^{k-1/2}) ^2}{(1-q^k) ^2},\quad\text{if }p\text{ is odd;}\nn
\ch L[p,r] &\le q^{h_{p,r}} (1 -q ^{\vert p\vert   })   \prod_{k=1} ^{\infty} \frac{ (1 + q^{k-1/2}) ^2}{(1-q^k) ^2},\quad\text{if }p\text{ is even.}\nonumber
\end{align}
\end{lemma}

\begin{proof}
We prove the claim for $p>0$. The case $p<0$ follows again by using contragredient modules.

Assume that $p>0$ is odd. We have constructed families of singular and subsingular vectors in Theorem \ref{sing-p-odd-1} of types
\bea
u_{p,r}^{(n)}&=&\Psi(-\tfrac{p}{2})(\a(-p))^n\vpr + \cdots,\nn
w_{p,r}^{(n)}&=&(\a(-p))^n\vpr + \cdots.\nonumber
\eea
Now we can eliminate all basis vectors (\ref{PBW-2}) containing either $\Psi^-(-\tfrac{p}{2})$ or $\alpha(-p)$ as a factor from the spanning set of $L[p,r]$. Therefore, the set of vectors
\[( \Psi^{+}  _{-\l^+} \Psi ^-  _{- \l^-} d_{-\mu^+ } c_{-\mu^- } )  \vpr,\qquad \l^+_i\ne\tfrac{p}{2},\ \mu^+_i\neq p\]
spans $L[p,r]$ which gives the claimed inequalities.

If $p>0$ is even, we have a family of singular vectors from Theorem \ref{sing-p-odd-1} of type
\[
u_{p,r}^{(n)}=(\a(-p))^n\vpr + \cdots
\]
so we can eliminate vectors (\ref{PBW-2}) containing $\alpha(-p)$ as a factor. Again, we obtain the wanted character inequality.
\end{proof}

Using these two lemmas, we get:

\begin{theorem} \label{xx} Assume that $p \in {\Z}\setminus \{ 0\} $ and $r \in {\C}$. Then we have:
\begin{align}
\ch L[p,r] &= q^{h_{p,r}} (1 -q ^{\tfrac{\vert p\vert}{2}})   \prod_{k=1} ^{\infty} \frac{ (1 + q^{k-1/2}) ^2}{(1-q^k) ^2},\quad\text{if }p\text{ is odd;}\nn
\ch L[p,r] &= q^{h_{p,r}} (1 -q ^{\vert p\vert   })   \prod_{k=1} ^{\infty} \frac{ (1 + q^{k-1/2}) ^2}{(1-q^k) ^2},\quad\text{if }p\text{ is even.}\nonumber
\end{align}
\end{theorem}

\section{The structure of $V[p,r]$}\label{struk2}
From Theorem \ref{xx} it follows that the $q$-character of the maximal submodule in $V[p,r]$ equals the $q$-character of the Verma module $V[p,r\pm1]$ if $p$ is even, and $V[p,r\pm\tfrac{1}{2}]$ if $p$ is odd (where the sign is opposite the sign of $p$). In three out of these four cases, this submodule is actually isomorphic to the Verma module.

\begin{theorem} \label{max-p-even}
Assume that $p\in\Z\setminus\{0\}$ is even. The maximal submodule of $V(h_{p,r},(1+p)\cla)$ is isomorphic to $V(h_{p,r}+\vert p\vert,(1+p)\cla)$.
\end{theorem}
\begin{proof}
From Theorem \ref{xx}, the maximality of the submodule $\<u_{p,r}\>$  in Corollary \ref{embed-verma-1} follows, which proves the claim for $p>0$.\\
Assume now that $p<0$. Let us f{}ix the PBW basis 
\bea\label{PBW-3}
(\Psi_{-\l^-} G_{-\l^+} \a_{-\mu^-} L_{-\mu^+} )\vpr
\eea
of Verma modules $V[p,r]$ and ordering (\ref{ordering-0}) with respect to $(\mu^+,\l^+)$. In Proposition \ref{sing_neg_p} we constructed a singular vector of the type 
\[
u_{p,r}= \Phi(p,r) \vpr = L(p)\vpr+\cdots
\]
where $\cdots$ denotes a sum $\sum w_i$ of basis elements (\ref{PBW-3}) of weight $-p$  such that $w_i<L(p)\vpr$. In particular, for each such monomial $w_i$ we have $\deg_{\mu^-,\l^-}w_i>0$. 
We show that $u_{p,r}$ is not annihilated  by any element of   $U(\SH^-)$, i.e.,  $U(\SH) u_{p,r} $ is isomorphic  to  a Verma module.
Note f{}irst that 
$$(\Psi_{-\l^-} G_{-\l^+} \a_{-\mu^-} L_{-\mu^+} ) L(p)\vpr = (\Psi_{-\l^-} G_{-\l^+} \a_{-\mu^-} L_{-\bar\mu^+} )\vpr + \cdots$$
where $\bar\mu^+=(\mu^+_1,\ldots,-p,\ldots\mu^+_l)$, and $\cdots$ denotes a sum of basis elements of equal $\deg_{\l^+,\mu^+}$ and lower $\ell_{\l^+,\mu^+}$. Let
$$0\ne X=\sum K_{\l^\pm\mu^\pm} (\Psi_{-\l^-} G_{-\l^+} \a_{-\mu^-} L_{-\mu^+} )\in\mathcal U(\SH^-)$$
be an arbitrary element of an universal enveloping algebra of $\SH^-$. Recall brackets from Def{}inition \ref{def}. We use the fact that if $w_i$ is a basis element as in (\ref{PBW-3}) such that $\deg_{\l^-,\mu^-}w_i>0$, then either $Xw_i=0$ or $Xw_i$ is a linear combination of basis (\ref{PBW-3}) monomials $w'_i$ such that $\deg_{\l^-,\mu^-}w'_i>0$. Then we have
$$Xu_{p,r} = XL(p)v_{p,r}+\cdots = \sum K_{\l^\pm\mu^\pm} (\Psi_{-\l^-} G_{-\l^+} \a_{-\mu^-} L_{-\bar\mu^+} )\vpr+\cdots$$
where $\cdots$ again denotes a sum $\sum w_i$ of monomials (\ref{PBW-3}) such that $w_i<(\Psi_{-\l^-} G_{-\l^+} \a_{-\mu^-} L_{-\bar\mu^+} )\vpr$. Therefore, $Xu_{p,r}\ne0$ and we conclude that $\SH^-$ acts freely on $u_{p,r}$, hence $\<u_{p,r}\>\cong V[p,r+1]$.
\end{proof}

\begin{remark}
In the proof of the previous Theorem, we actually show that $\Phi(p,r)$ is not a zero divisor. In order to prove this, we use an explicit expression for $\Phi(p,r)$. But we hope that this is a part of a more general theory (cf.\ \cite{G2020}). 
\end{remark}
From the previous theorem we conclude that Theorem \ref{sing-p-even-1} lists all singular (and subsingular) vectors in $V[p,r]$, $p\in\N$ even, and we get a chain of submodules
$$V[p,r-i-1]\subseteq V[p,r-i],\qquad i\in\Zp.$$
Similarly, for $p\in\Z_{<0}$ even we have a chain
$$V[p,r+i+1]\subseteq V[p,r+i],\qquad i\in\Zp.$$

\begin{theorem}  \label{max-p-odd-manji}
The maximal submodule in $V[p,r]$, $p>0$ odd is generated by a subsingular vector $w_{p,r}$ given by (\ref{subsing}).
\end{theorem}
\begin{proof}
Let $M=\<w_{p,r}\>$ and consider $V[p,r]/M$. In the proof of Lemma \ref{nejed2} we have shown that $$\ch V[p,r]/M \leq q^{h_{p,r}} (1 -q ^{\tfrac{\vert p\vert}{2}})   \prod_{k=1} ^{\infty} \frac{ (1 + q^{k-1/2}) ^2}{(1-q^k) ^2}.$$
But, by Theorem \ref{xx}, the right hand side is equal to $\ch L[p,r]$. Since $\ch V[p,r]/M\geq\ch L[p,r]$, we conclude that $V[p,r]/M=L[p,r]$ i.e.\ $M$ is the maximal submodule.
\end{proof}

\begin{remark}
It can be shown that the maximal submodule in $V[p,r]$, $p<0$ odd is isomorphic to $V[p,r+\tfrac{1}{2}]$ and generated by a singular vector of the type $u_{p,r}=G(\tfrac{p}{2})\vpr+\cdots$.
\end{remark}
Structure of Verma modules in case $p\in\Z_{<0}$ is shown in Figure \ref{fig_verma_neg}. Free f{}ield realisation of Verma modules for $p\in\N$ odd is shown in Figure \ref{verma_fig1}.

\appendix
\label{dodatak}
\section{Figures}
Below we present embedding diagrams for the Verma module $V[p,r]$ and its dual $\mathcal F_{-p,-r}$, $p\in\N$. We should mention that in the present paper we don't present all proofs for the structure of $V[p,r]$ for $p\in\Z_{>0}$ odd, since it requires a very subtle analysis of indecomposable modules. In our forthcoming papers we shall investigate in more details the indecomposable modules and tensor categories related to these structures (cf.\ \cite{AR3}).

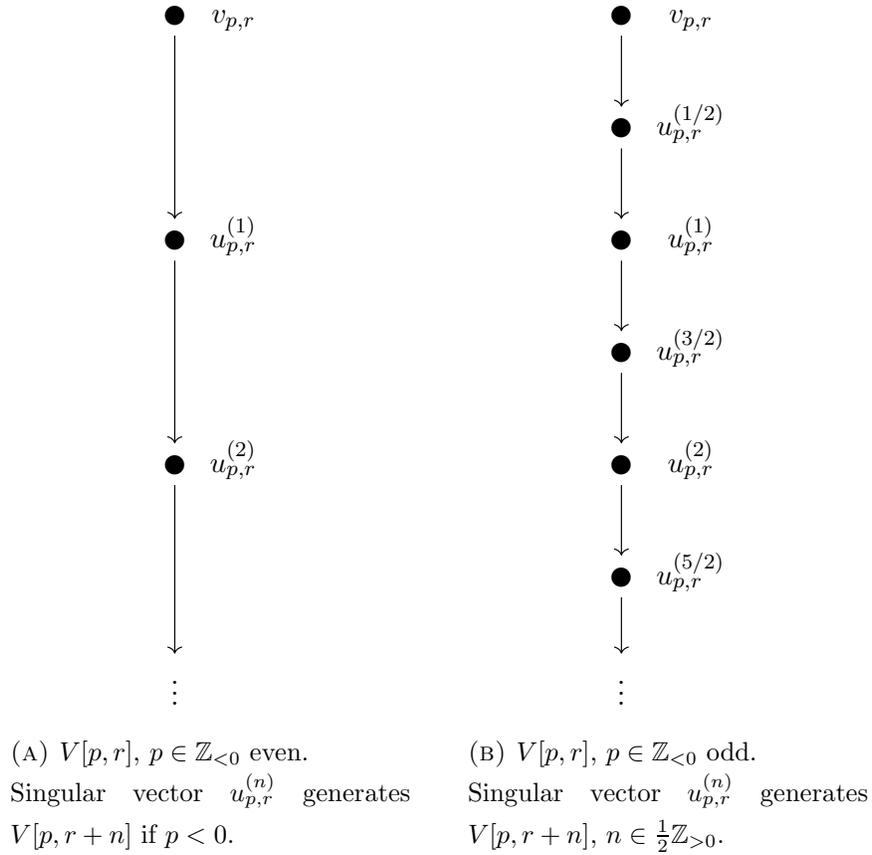
\begin{figure}[H]
    \centering
    \begin{subfigure}{.42\textwidth}
    \centering
    \begin{tikzcd}[column sep=-0.7em]
        \bcr\arrow[dd]&\ \ \vpr\\
        &\phantom{u^{1/2}_{p,r}}\\
        \bcr\arrow[dd]&\ \ u^{(1)}_{p,r}\\
        &\phantom{u^{3/2}_{p,r}}\\
        \bcr\arrow[dd]&\ \ u^{(2)}_{p,r}\\
        &\phantom{u^{5/2}_{p,r}}\\
        \vdots&
    \end{tikzcd}
    \caption{$V[p,r]$, $p\in\Z_{<0}$ even.\\ Singular vector $u^{(n)}_{p,r}$ generates $V[p,r+n]$ if $p<0$.}
    \end{subfigure}\qquad
    \begin{subfigure}{.42\textwidth}
    \centering
    \begin{tikzcd}[column sep=-0.7em]
        \bcr\arrow[d]&\ \ \vpr\\
        \bcr\arrow[d]&\ \ u^{(1/2)}_{p,r}\\
        \bcr\arrow[d]&\ \ u^{(1)}_{p,r}\\
        \bcr\arrow[d]&\ \ u^{(3/2)}_{p,r}\\
        \bcr\arrow[d]&\ \ u^{(2)}_{p,r}\\
        \bcr\arrow[d]&\ \ u^{(5/2)}_{p,r}\\
        \vdots&
    \end{tikzcd}
    \caption{$V[p,r]$, $p\in\Z_{<0}$ odd.\\ Singular vector $u^{(n)}_{p,r}$ generates $V[p,r+n]$, $n\in\tfrac{1}{2}\N$.}
    \end{subfigure}
    \caption{Structure of Verma modules for $p<0$.}
    \label{fig_verma_neg}
\end{figure}

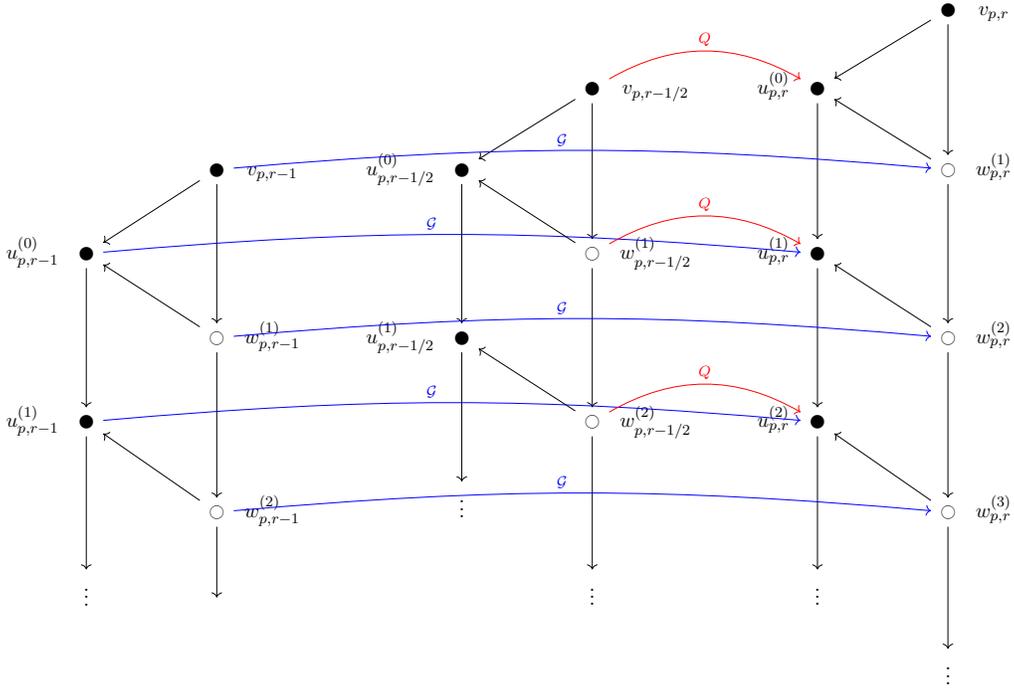
\begin{figure}[H]
    \centering
    \adjustbox{scale=.7,center}{\begin{tikzcd}
&&&&&    &&&&& &&&\bcr\arrow[dll]\arrow[dd]&\hspace{-25pt}v_{p,r}\\
&&&&&    &&&\bcr\arrow[dll]\arrow[dd]\arrow[rrr,red,bend left,"Q"]&\hspace{-25pt}v_{p,r-1/2} &u^{(0)}_{p,r}\hspace{-25pt}&\bcr\arrow[dd]&&&\\
&&&\bcr\arrow[dll]\arrow[dd]\arrow[rrrrrrrrrr,blue,bend left=5,"\mathcal G\qquad"]&\hspace{-25pt}v_{p,r-1}&u^{(0)}_{p,r-1/2}\hspace{-25pt}&\bcr\arrow[dd]&&&
&&&&\ec\arrow[llu]\arrow[dd]&\hspace{-25pt}w_{p,r}^{(1)}\\
u^{(0)}_{p,r-1}\hspace{-25pt}&\bcr\arrow[dd]\arrow[rrrrrrrrrr,blue,bend left=5,"\mathcal G\qquad"]&&&&    &&&\ec\arrow[llu]\arrow[dd]\arrow[rrr,red,bend left,"Q"]&\hspace{-25pt}w_{p,r-1/2}^{(1)}&u^{(1)}_{p,r}\hspace{-25pt}&\bcr\arrow[dd]&&&\\
&&&\ec\arrow[dd]\arrow[llu]\arrow[rrrrrrrrrr,blue,bend left=5,"\mathcal G\qquad"]&\hspace{-25pt}w^{(1)}_{p,r-1}&u^{(1)}_{p,r-1/2}\hspace{-25pt}&\bcr\arrow[dd]&&&
&&&&\ec\arrow[llu]\arrow[dd]&\hspace{-25pt}w_{p,r}^{(2)}\\
u^{(1)}_{p,r-1}\hspace{-25pt}&\bcr\arrow[dd]\arrow[rrrrrrrrrr,blue,bend left=5,"\mathcal G\qquad"]&&&&    &&&\ec\arrow[llu]\arrow[dd]\arrow[rrr,red,bend left,"Q"]&\hspace{-25pt}w_{p,r-1/2}^{(2)}&u^{(2)}_{p,r}\hspace{-25pt}&\bcr\arrow[dd]&&&\\
&&&\ec\arrow[llu]\arrow[d]\arrow[rrrrrrrrrr,blue,bend left=5,"\mathcal G\qquad"]&\hspace{-25pt}w^{(2)}_{p,r-1}&&\vdots&&&  &&&&\ec\arrow[llu]\arrow[dd]&\hspace{-25pt}w_{p,r}^{(3)}\\
&\vdots&&\ &&    &&&\vdots&&&\vdots\\
&&&&&&&&&&&&&\vdots
    \end{tikzcd}}
    \caption{Realisation of $\mathcal F_{p,r}\cong V[p,r]$, $p>0$ odd.}
    \label{verma_fig1}
\end{figure}

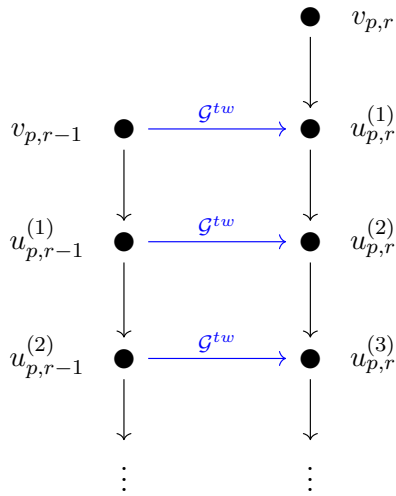
\begin{figure}[H]
\centering
    \begin{tikzcd}
        &&&\bcr\arrow[d]&\hspace{-25pt}\vpr\\
        v_{p,r-1}\hspace{-25pt}&\bcr\arrow[d]\arrow[rr,blue,"\mathcal G^{tw}"]&&\bcr\arrow[d]&\hspace{-25pt}u^{(1)}_{p,r}\\
        u^{(1)}_{p,r-1}\hspace{-25pt}&\bcr\arrow[d]\arrow[rr,blue,"\mathcal G^{tw}"]&&\bcr\arrow[d]&\hspace{-25pt}u^{(2)}_{p,r}\\
        u^{(2)}_{p,r-1}\hspace{-25pt}&\bcr\arrow[d]\arrow[rr,blue,"\mathcal G^{tw}"]&&\bcr\arrow[d]&\hspace{-25pt}u^{(3)}_{p,r}\\
        &\vdots&&\vdots&
    \end{tikzcd}
\caption{Realisation of $\mathcal F_{p,r}\cong V[p,r]$, $p>0$ even.}
\label{ff_paran_poz}
\end{figure}

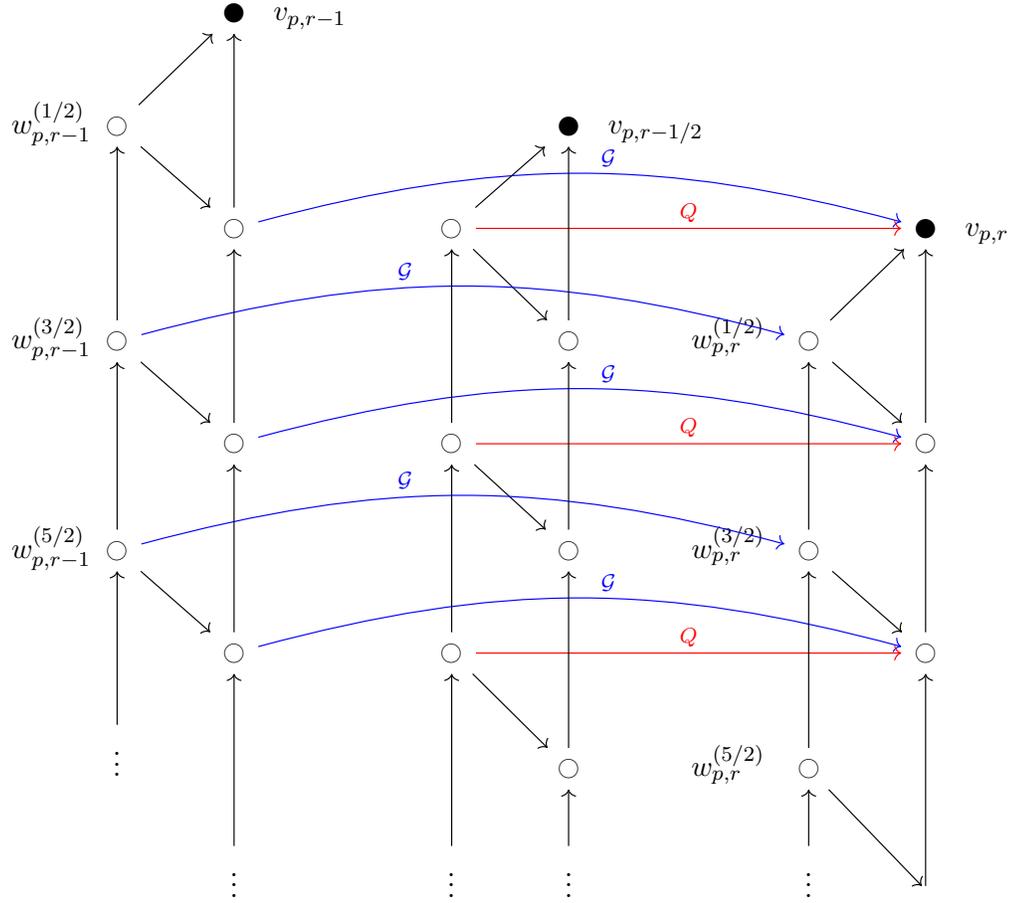
\begin{figure}[H]
   \centering
   \begin{tikzcd}
&&\bcr&\hspace{-25pt}v_{p,r-1}\\
w^{(1/2)}_{p,r-1}\hspace{-30pt}&\ec\arrow[ur]\arrow[dr]&&&&\bcr&\hspace{-25pt}v_{p,r-1/2}\\
&&\ec\arrow[uu]\arrow[rrrrrr,blue,bend left=15,"\qquad\mathcal G"]&&\ec\arrow[ur]\arrow[dr]\arrow[rrrr,red,"Q"]&&    &&\bcr&\hspace{-25pt}\vpr\\
w^{(3/2)}_{p,r-1}\hspace{-30pt}&\ec\arrow[uu]\arrow[dr]\arrow[rrrrrr,blue,bend left=15,"\mathcal G\qquad\qquad"]&& &&\ec\arrow[uu]&    w^{(1/2)}_{p,r}\hspace{-30pt}&\ec\arrow[ur]\arrow[dr]\\
&&\ec\arrow[uu]\arrow[rrrrrr,blue,bend left=15,"\qquad\mathcal G"]&&\ec\arrow[uu]\arrow[dr]\arrow[rrrr,red,"Q"]&&    &&\ec\arrow[uu]\\
w^{(5/2)}_{p,r-1}\hspace{-30pt}&\ec\arrow[uu]\arrow[dr]\arrow[rrrrrr,blue,bend left=15,"\mathcal G\qquad\qquad"]&& &&\ec\arrow[uu]&    w^{(3/2)}_{p,r}\hspace{-30pt}&\ec\arrow[uu]\arrow[dr]\\
&&\ec\arrow[uu]\arrow[rrrrrr,blue,bend left=15,"\qquad\mathcal G"]&& \ec\arrow[uu]\arrow[dr]\arrow[rrrr,red,"Q"]&&    &&\ec\arrow[uu]\\
&\vdots\arrow[uu]&& &&\ec\arrow[uu]&    w^{(5/2)}_{p,r}\hspace{-30pt}&\ec\arrow[uu]\arrow[dr]\\
&&\vdots\arrow[uu]& &\vdots\arrow[uu]&\vdots\arrow[u]&    &\vdots\arrow[u]&\ \arrow[uu]
    \end{tikzcd}
    \caption{$\mathcal F_{p,r}\cong V[-p,-r]^*$, $p<0$ odd.\\$\mathcal G w_{p,r-1}^{(n+1/2)}=w_{p,r}^{(n-1/2)}$ and $Q w_{p,r-1/2}^{(n+1/2)}$ are subsingular vectors.}
    \label{ff_neparan_neg}
\end{figure}

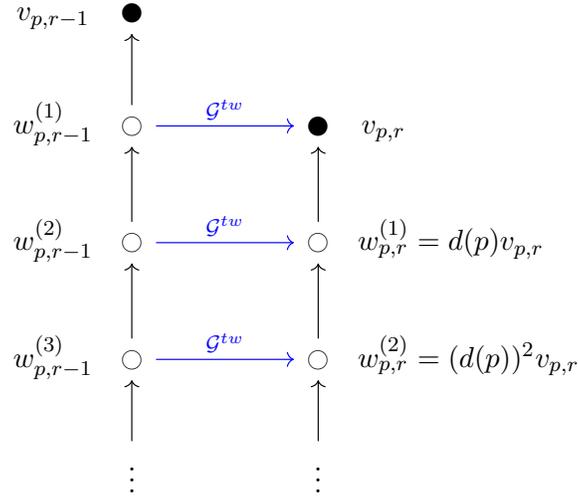
\begin{figure}[H]
    \centering
    \begin{tikzcd}
    v_{p,r-1}\hspace{-25pt}&\bcr\\
    w^{(1)}_{p,r-1}\hspace{-25pt}&\ec\arrow[u]\arrow[rr,blue,"\mathcal G^{tw}"]&&\bcr&\hspace{-25pt}\vpr\phantom{w^{(1)}_{p,r+1}=d(p)}\\
    w^{(2)}_{p,r-1}\hspace{-25pt}&\ec\arrow[u]\arrow[rr,blue,"\mathcal G^{tw}"]&&\ec\arrow[u]&\hspace{-25pt}w^{(1)}_{p,r}=d(p)\vpr\phantom{()^2}\\
    w^{(3)}_{p,r-1}\hspace{-25pt}&\ec\arrow[u]\arrow[rr,blue,"\mathcal G^{tw}"]&&\ec\arrow[u]&\hspace{-25pt}w^{(2)}_{p,r}=(d(p))^2v_{p,r}\\
    &\vdots\arrow[u]&&\vdots\arrow[u]
    \end{tikzcd}
    \caption{$\mathcal F_{p,r}\cong V[-p,-r]^*$, $p<0$ even.\\ Subsingular vector $w^{(n-1)}$ generates $\Ker_{\mathcal F_{p,r}}(\mathcal G^{tw})^n$, $n\in\N$.
    }
    \label{ff_paran_neg}
\end{figure}

\end{document}